\documentclass[a4paper,12pt]{article}

\usepackage{amssymb,amsmath}
\usepackage{xcolor}
\usepackage{amsthm}
\usepackage{enumerate} 
\usepackage{stmaryrd}

\newtheorem{theo}{Theorem}[section]
\newtheorem{coro}[theo]{Corollary}

\newcommand{\R}{{\bf R}}

\newcommand{\p}{\partial}

\newcommand{\lsm}{\lesssim}
\newcommand{\ep}{\varepsilon}

\newcommand{\al}{\alpha}
\newcommand{\bt}{\beta}
\newcommand{\dt}{\delta}

\newcommand{\F}{{\cal{F}}}

\newcommand{\dsp}{\displaystyle}

\newcommand{\nb}{\nabla}
\newcommand{\fa}{\frac}

\newcommand{\lp}[1]{\left(#1 \right)}
\newcommand{\lb}[1]{\left\{#1 \right\}}
\newcommand{\lbt}[1]{\left[#1 \right]}

\newcommand{\gb}[1]{\llbracket #1 \rrbracket}

\begin{document}
 
\begin{center}
{\large\bf 
Existence of solutions of semilinear wave equations with time-dependent propagation speed and time derivative nonlinearity
}
\end{center}

\vspace{3mm}
\begin{center}

Kimitoshi Tsutaya$^\dagger$ 
        and Yuta Wakasugi$^\ddagger$ \\
\vspace{1cm}

 $^\dagger$Graduate School  of Science and Technology \\
Hirosaki University  \\
Hirosaki 036-8561, Japan
\footnotetext{
{\bf Funding information}: 
The research was supported by JSPS KAKENHI Grant Number A22H000971. }

\vspace{5mm}
        $^\ddagger$
Graduate School of Advanced Science and Engineering \\
Hiroshima University \\
Higashi-Hiroshima, 739-8527, Japan
\end{center}

\begin{abstract}
Consider wave equations with time derivative nonlinearity and time-dependent propagation speed which are generalized versions of the wave equations in the Friedmann-Lema\^itre-Robertson-Walker \\
 (FLRW) spacetime, the de Sitter spacetime and the anti-de Sitter space time. 
We show lower bounds of the lifespan of solutions as well as the global existence by providing an integrability condition on the propagation speed function, which is applicable to the nonlinear wave equation in the expanding FLRW spacetime including the de Sitter spacetime. 
We also prove that blow-up in a finite time occurs for the generalized form of the equation in contracting universes such as the anti-de Sitter spacetime, 
as well as upper bounds of the lifespan of blow-up solutions.  
    
\end{abstract}


\section{Introduction}

\addtolength{\baselineskip}{2mm}

We consider the following Cauchy problem: 
\begin{equation}
\begin{cases}
\dsp
u_{tt}-a(t)^2\Delta u+b(t)u_t=F(u, u_t, \nb_x u), &\qquad  t>0, \; x\in \R^n,   \\
u(0,x)=\ep u_0(x), \; u_t(0,x)=\ep u_1(x), &\qquad x\in \R^n,  
\end{cases}
\label{eq-0}
\end{equation}
where $F=|u|^p, \; |u_t|^p$ or $|\nb_xu|^p$ with $p>1$, 
$\Delta=\p_1^2+\cdots +\p_n^2, \; \nb_xu = (\p_1 u, \cdots,\p_n u), \; \p_j=\p/\p x^j, \; j=1,\cdots,n, 
 \; (x^1,\cdots, x^n)\in \R^n$,  and $\ep>0$ is a small parameter.  

In \cite{TW1}-\cite{TW6} 
we have studied blow-up solutions of semilinear wave equations in the spatially flat FLRW (Friedmann-Lema\^itre-Robertson-Walker) spacetimes, which are of the form Eq. in  \eqref{eq}. 
Typical coefficients $a(t)$ and $b(t)$ are $(a(t),b(t))=((t+1)^{-\al}, \mu/(t+1))$ in the FLRW spacetimes, and $(a(t),b(t))=(e^{-Ht},\mu)$ in the de Sitter spacetime, where $\al\ge 0, \; \mu\ge 0$ and $H>0$ are constants.  

The remarkable result among these is that 
if $F=|u|^p$ or $|\nb_xu|^p$ and $(a(t),b(t))=((t+1)^{-\al}, \mu/(t+1))$ or $(a(t),b(t))=(e^{-Ht},\mu)$ with $\al>1$, $\mu\ge 0$ and $H>0$, then blow-up in a finite time can occur for all $p>1$. 

In \cite{TW7} it is shown that for generalized $a(t)$ and $b(t)$, blow-up in a finite time for \eqref{eq-0} occurs provided that 
\begin{align*}
&a,b\in C^1([0,\infty)), \; a(t)>0,\;
0<C_1\le b(t)\le C_2, \; \limsup_{t\to\infty}\fa{|\dot b(t)|}{b(t)^2}<1,  \\
\intertext{and } 
& \begin{cases} 1-n\bt(p-1)>0 & \mbox{for }F=|u|^p, \\
 1-\bt\lb{n(p-1)+p}>0 \quad & \mbox{for }F= |\nb_xu|^p,
  \end{cases}
\end{align*}
where $\bt=\limsup_{t\to \infty}(\ln A(t)/\ln t)$ and $\dsp A(t)=\int_0^t a(s)ds$. 
We note that $\bt =0$ for $a\in L^1([0,\infty)$, which leads to blow-up for all $p>1$. 

In contrast to these nonlinear terms, for the time derivative nonlinearity, circumstances are different. 
It is shown by \cite{CFO} that for  $F=u_t^2$, there exist global solutions of \eqref{eq} with small initial data if $a>0, \; \dot a=da/dt<0, \; a\in L^1([0,\infty))$ and $b=-n\dot a/a$. 
We refer to \cite{CFO} for other quadratic nonlinearity. 

For more general $F=|u_t|^p$ with an integer $p\ge 2$, global existence of solutions is proved by \cite{WY} under the conditions  
$a\in C^\infty([0,\infty)), \; a>0, \; \dot{(a^{-1})}>0, \;\ddot{(a^{-1})}>0$, $b=-n\dot a/a$,   
and  the integrability conditions of $\fa{d^k}{dt^k}(\dot a/a)$ and $a^{-1}\fa{d^k}{dt^k}(a\dot a)$ for $1\le k\le L$ and $L\ge n+4$.

In the present paper we first show global existence of solutions of \eqref{eq} with $F=|u_t|^p$ under weaker conditions on $a$ and $b$. 
Especially, this result provides that if $n=1,2$ and $a^{p-1}\in L^1([0,\infty))$, then there exist global solutions of \eqref{eq} with small initial data for all $p>1$. 
If we go back to the original wave equation in the FLRW spacetime, that is $\Box_g u=-|u_t|^p$ where $\Box_g$ is the d'Alembertian in the FLRW spacetime, 
this global existence result corresponds to the case of the expanding universe.

The rest part of the paper is concerned with the case of the contracting universe 
like the anti-de Sitter spacetime. We remark that in this case the coefficient $b(t)$ becomes non-positive. 
As the second main result, we show blow-up in a finite time of solutions for the time derivative nonlinearity and more general $a(t)$ and $b(t)$ than the case of the anti-de Sitter spacetime, following \cite[Proposition 1]{Na}.
Whereas \cite[Proposition 1]{Na} deals with large initial data, our study focuses on small initial data. 
For this reason, it is necessary to distinguish between cases according to the behavior of $a(t)$ and $b(t)$. 

Our result shows that the lifespan of blow-up solutions is quite short as estimated by $\ln 1/\ep$ in some case.

This paper is organized as follows. 
In Section 2, we first give a local existence result for \eqref{eq-0} with $F=|u_t|^p$, then  state a global existence theorem.   
We use the energy method to prove the theorem in terms of the integrability of $a(t)^{p-1}$ following \cite{GT,IT}. 
We also show a lower bound of the lifespan of local solutions in time if the existence of global solutions  is not ensured. In addition, we apply the theorem to the equation in the flat FLRW spacetime and discuss the critical exponent of $p$ for the existence of global solutions in some cases. 
In Section 3, we consider the non-positive coefficient $b(t)$. Our second main result presents that upper bounds of the lifespan of the blow-up solutions 
differ depending on the relationship on $a(t)$ and $b(t)$. 


\section{Global Existence}
\setcounter{equation}{0}

Consider the Cauchy problem 
\begin{equation}
\begin{cases}
\dsp
u_{tt}-a(t)^2\Delta u+b(t)u_t=F(u_t), &\qquad  t>0, \; x\in \R^n,   \\
 & \\
u(0,x)=f(x), \; u_t(0,x)=g(x), &\qquad x\in \R^n,  
\end{cases}
\label{eq}
\end{equation}
and assume that $a\in C^1([0,\infty))$ and $b\in C([0,\infty))$ satisfy
\begin{equation}
a(t)>0 \quad \mbox{and }\quad \dot a +ab\ge 0 \quad \mbox{ for }t\ge 0, 
\label{ab-asm}
\end{equation}
where $\dot a=da/dt$.

Let $p>1$ and $\gb{p}$ denote the largest integer smaller than $p$. 
We also assume that 
there exist $p>1$ and some  constant $C>0$  
such that 
$F(v)$ satisfies
\begin{equation}
\begin{cases}
F\in C^{\gb{p}}(\R)
, \\
|F^{(j)}(v)|\le C|v|^{p-j} \quad  \mbox{ for }0\le j\le \gb{p},   \\
|F^{(\gb{p})}(v_1)-F^{(\gb{p})}(v_2)|\le C|v_1-v_2|^{p-\gb{p}}. 
\end{cases}
\label{F}
\end{equation}
We see from \eqref{F} that $F$ satisfies $F^{(j)}(0)=0$ for $0\le j\le \gb{p}$. 


We write $|u|_\infty=\|u\|_{L^\infty}, \; \|u\|_p=\|u\|_{L^p}$ with $p\ge 1$ and define  
$\|u\|_{H^s}=\|(1-\Delta)^{s/2}u\|_2$, 
$|\nb|^su=\F^{-1}[|\xi|^s\F u]$ for $s\in\R$, where $\F$ is the Fourier transform. 

Let 
\[
E[u](t)^2=\int (a(t)^{-2}u_t^2+|\nb_xu|^2)dx. 
\]
If $u$ is a solution of \eqref{eq}, then 
\begin{align*}
\fa 12\fa d{dt}E[u](t)^2
&= \int(-a^{-3}\dot a u_t^2+a^{-2}u_tu_{tt}+\nb_xu\cdot \nb_xu_t)dx \\
&= \int(-a^{-3}\dot a u_t^2+u_t(a^{-2}u_{tt}-\Delta u))dx \\
&= \int(a^{-2}u_tF-a^{-3}(\dot a+ab)u_t^2)dx. 
\end{align*}
Using \eqref{ab-asm} and the Schwarz inequality, we have 
\[
\fa 12\fa d{dt}E[u](t)^2 \le \int a^{-2}u_tF dx \le \|a^{-1}F\|_2 E[u](t). 
\]
Thus, perfoming the differentiation on the left-hand side and integrating the both sides from $0$ to $t$, we obtain
\begin{equation}
E[u](t)\le E[u](0)+\int_{0}^t \|a(\tau)^{-1}F(\tau)\|_2d\tau. 
\label{Energyine}
\end{equation}
Similarly, for $|\nb|^su$
\begin{equation}
E[|\nb|^su](t)\le E[|\nb|^su](0)+\int_{0}^t \|a(\tau)^{-1}|\nb|^sF(\tau)\|_2d\tau. 
\label{sEnergyine}
\end{equation}
Setting
\begin{equation}
E_s[u](t)^2=\|a(t)^{-1}u_t(t)\|_{H^s}^2+\|\nb_xu\|_{H^s}^2, 
\label{sEnergy}
\end{equation}
we see from \eqref{Energyine} and \eqref{sEnergyine} that 
\begin{equation}
E_s[u](t)\le E_s[u](0)+C\int_{0}^t \|a(\tau)^{-1}F(\tau)\|_{H^s}d\tau. 
\label{sEine}
\end{equation}

We set for $0<T\le \infty$, 
\begin{equation}
\begin{split}
X_T^s&=C([0,T):H^{s+1}(\R^n))\cap C^1([0,T):H^{s}(\R^n)), \\
X_{M,T}&=\{u\in X_T^s: \|u\|_{X_T^s}\le M\}, \\
\|u\|_{X_T^s}&=\sup_{0\le t\le T}E_s[u](t), \\
d(u,v)&=\|u-v\|_{X_T^0}. 
\end{split}
\label{XTs}
\end{equation} 


We first state a local existence theorem. 
\begin{theo}
Let $n\ge 1$, $p>1$, 
and let $a$ and $b$ 
satisfy \eqref{ab-asm}. 
Assume that $F$ satisfies \eqref{F} with $p>n/2$ and 
$n/2<s<p$. 
In particular, if $F$ is given by $F(v)=cv^p$ with an integer $p>1$ and some constant $c$,  assume only $s>n/2$. 
Suppose that 
$f\in H^{s+1}(\R^n)$ and $g\in H^s(\R^n)$. 
Then there exists a $T>0$ such that the problem (2.1) has  a unique solution in $ X_T^s$. 

\label{th21}
 
\end{theo}

\noindent
{\it Proof. }
Consider the problem \eqref{eq} for $0\le t<T$. 
For $v\in X_{M,T}$, we define a map $u=\Phi(v)$, where $u$ is a unique solution of \eqref{eq} with $F(u_t)$ replaced by $F(v_t)$ for $0\le t<T$. 
We first show that $\Phi$ maps $X_{M,T}$ into itself. 
From \eqref{sEine}, 
\begin{equation}
E_s[u](t)\le E_s[u](0)+C\int_{0}^t \|a^{-1}|v_t|^p(\tau)\|_{H^s}d\tau \quad \mbox{for }
v\in X_{M,T}. 
\label{Es}
\end{equation}
Choosing $s$ such that $n/2<s<p$, we have by \cite[Lemma 2.3]{HW}, \cite[Theorem 3.4]{LiPo}, Sobolev's embedding and \eqref{ab-asm},  
\begin{equation}
\|a^{-1}|v_t|^p(\tau)\|_{H^s}
\lsm |v_t|_\infty^{p-1}\|a^{-1}v_t\|_{H^s} 
\lsm a(\tau)^{p-1}\|a^{-1}v_t\|_{H^s}^p
\le CM^p a(\tau)^{p-1}. 
\label{Fs}
\end{equation}
In particular, if $F$ is given by $F(v)=cv^p$ with an integer $p>1$ and some constant $c$,  we have only to take $s$ satisfying $s>n/2$ and see that \eqref{Fs} holds. 

By \eqref{Es} and \eqref{Fs}, 
\[
\|u\|_{X_T^s}\le E_s[u](0)+CM^p \int_{0}^T a(\tau)^{p-1}d\tau. 
\]
Let $E_s[u](0)\le M/2$ by taking $M>0$ large enough. 
Choosing $T>0$ so small that $CM^{p-1}\int_{0}^T a(\tau)^{p-1}d\tau\le 1/2$, we obtain 
\[
\|u\|_{X_T^s}\le M. 
\]
Thus, $\Phi$ maps $X_{M,T}$ into itself. 

We next show that $\Phi$ is a contration mapping if $T$ is small. 
Let $u=\Phi(v)$ and $\tilde u=\Phi(\tilde v)$ for $v,\tilde v\in X_{M,T}$. 
Then $u-\tilde u$ satisfies
\[
E[u-\tilde u](t)\le \int_{0}^t\|a^{-1}(F(v_t)-F(\tilde v_t))\|_2d\tau. 
\]
From the condition \eqref{F},  
we have by the mean value theorem, Sobolev's embedding and \eqref{ab-asm},  
\begin{align*}
E[u-\tilde u](t)&\lsm
\int_{0}^t \|a^{-1}(v_t-\tilde v_t)\|_2(|v_t|_\infty+|\tilde v_t|_\infty)^{p-1}
d\tau \\
&\lsm 2^{p-1}M^{p-1}\|v-\tilde v\|_{X^0_T}\int_{0}^T a(\tau)^{p-1}d\tau, 
\end{align*}
hence, $d(u,\tilde u)\le 2^{p-1}CM^{p-1}d(v,\tilde v)\int_{0}^T a(\tau)^{p-1}d\tau$. 
Taking $T$ so small that $2^{p-1}CM^{p-1}\int_{0}^T a(\tau)^{p-1}d\tau\le 1/2$, we see that the mapping $\Phi$ is a contration. 

Since the space $(X_T^s,d)$ is a complete metric space, it follows that 
$\Phi$ has a unique fixed point, thus the problem \eqref{eq} has a unique local solution in $X_T^s$ if $T$ is sufficiently small. 
This completes the proof of Theorem \ref{th21}. 
\qed

\vspace{1cm}

We now consider the lifespan of time local solutions of \eqref{eq}. 
Let us introduce a parameter $\ep>0$ for the initial data $f$ and $g$ in \eqref{eq} such as $f(x)=\ep u_0(x)$ and $g(x)=\ep u_1(x)$. 

Let $T_\ep$ be the lifespan of solutions of \eqref{eq}, that is, $T_\ep$ is the supremum of T such that \eqref{eq} has a solution for $x\in \R^n$ and $0\le t<T$.

\begin{theo}
Let the assumptions of Theorem \ref{th21} hold and 
let $u(t,x)$ be a unique local solution provided by Theorem \ref{th21}, and also
\[
A_{p-1}(t)=\int_{0}^t a(\tau)^{p-1}d\tau. 
\]
\vspace{-1cm}
\begin{enumerate}[(1)]  
\item 
If $a^{p-1}\in L^1([0,\infty))$ and $\ep$ is sufficiently small, then the lifespan $T_\ep$ of $u$ satisfies $T_\ep=\infty$, i.e., 
the problem \eqref{eq} has  a unique global solution $u\in C([0,\infty):H^{s+1}(\R^n))\cap C^1([0,\infty):H^{s}(\R^n))$ 
satisfying 
\[
\sup_{t\ge 0}(\|a(t)^{-1}u(t)\|_{H^s}+\|\nb_xu\|_{H^s})<\infty. 
\]
\item 
If $a^{p-1}\not\in L^1([0,\infty))$,  
then there exists a constant $\ep_0>0$ sufficiently small such that 
for all $0<\ep\le \ep_0$, the lifespan $T_\ep$ of $u$ satisfies 
\begin{equation}
A_{p-1}(T_\ep)\ge C\ep^{-(p-1)} \quad \mbox{with some constant }C>0.  
\label{Aep}
\end{equation}

\end{enumerate}

\label{th22}
 
\end{theo}

\noindent
{\bf Remark. }
From the theorem above, we see that 
if $n=1,2$, then global solutions of \eqref{eq} exist 
for all $p>1$. 

\noindent
{\it Proof. }
In the previous proof we replace $M$ with $C_d\ep$, 
where $\|u_0\|_{H^{s+1}}+\|u_1\|_{H^s}\le C_d/2$, and 
set for $X_T^s$
\begin{equation}
\begin{split}
X_0&=\{u\in X_T^s: \|u\|_{X_{a,p}^s}\le  C_d\ep \}, \\
\|u\|_{X_{a,p}^s}&=\sup_{t\in I_{a,p}}E_s[u](t), \\
I_{a,p}&=
\begin{cases}
[0,\infty) & \mbox{ if }a^{p-1}\in L^1([0,\infty)), \\
[0,T) & \mbox{ otherwise  }, 
\end{cases}
\end{split}
\label{X0}
\end{equation}

\begin{equation}
C_{a,p}=
\begin{cases}
\dsp \int_{0}^\infty a(\tau)^{p-1}d\tau & \mbox{ if }a^{p-1}\in L^1([0,\infty)), \\
A_{p-1}(T) & \mbox{otherwise}.
\end{cases}
\label{Cap}
\end{equation}
We see from the poof of Theorem \ref{th21} that 
the mapping $\Phi$ is a contration from $X_0$ into itself, 
by choosing $\ep$ so small that 
\begin{equation}
2^{p-1}C C_d^{p-1}\ep^{p-1}C_{a,p}\le \fa 12.  
\label{epC}
\end{equation}
Thus, it follows from \eqref{Cap} and \eqref{epC} that 
if $a^{p-1}\in L^1([0,\infty))$, then there exists a unique global solution of \eqref{eq} since $C_{a,p}$ is a constant. 

On the other hand, if $a^{p-1}\not\in L^1([0,\infty))$, then we obtain \eqref{Aep} by \eqref{Cap} and \eqref{epC}. 
This completes the proof of Theorem \ref{th22}. 
\hfill\qed

\vspace{5mm}

In the end of this section, we discuss specific coefficients $a(t)$ and $b(t)$ which are applicable to Theorem \ref{th22} by considering the equation in the flat FLRW spacetime.

If the scale factor of the expanding FLRW metric is $(t+1)^\al$, then 
$\Box_g u=-F(u_t)$ is of the form \eqref{eq} with $(a(t),b(t))=((t+1)^{-\al},\mu/(t+1))$, where $\al>0$ and $\mu\ge 0$ are constants. See \cite{TW4} for details. 

We see from the theorem above that global solutions of \eqref{eq} exist for small initial data if $p>1+1/\al$ and $\mu\ge \al$.  
If $1<p\le 1+1/\al$ and $\mu\ge \al$, we obtain from \eqref{Aep} the lower bounds of the lifespan $T_\ep$ 
\begin{align*}
T_\ep &\ge C\ep^{\fa{-(p-1) }{1-\al(p-1)}} &&\mbox{if }1<p< 1+\fa 1\al, \\
T_\ep &\ge \exp C\ep^{-(p-1)} &&\mbox{if }p= 1+\fa 1\al, 
\end{align*}
where $C$ is a constant independent of $\ep$. 

On the other hand, for this pair $(a(t),b(t))$, the following blow-up conditions are shown in \cite{TW4}
\footnote{
In Theorems 2.1 and 2.3 of \cite{TW4} $n\ge 2$ is assumed as a condition, however both of these theorems hold also for $n=1$. This fact is ensured by changing the test function $\phi$ in  the proof of Theorem 2.1, and by noting that \cite[Lemma 2.2]{TW4} and the property of finite speed of propagation do not rely on the space dimension (see also \cite{HHP}). }:
\begin{align*}
&1<p<1+\fa 1\mu \qquad \mbox{if }\al\ge 1, \\
&\begin{cases}\dsp 1<p\le 1+\fa 2{(1-\al)(n-1)+\mu+\al} \\
 \quad \mbox{or }\\
 \dsp 1<p< 1+\fa {1}{n(1-\al)+\mu}
\end{cases} \qquad \mbox{if }0\le \al< 1. 
\end{align*}
The upper bounds of the lifespan in the case $\al\ge 1$ are 
\begin{align*}
&T_\ep^{1-\mu(p-1)}(\ln T_\ep)^{-n(p-1)}\le C\ep^{-(p-1)} &&\mbox{if }\al=1, 
\\
&T_\ep\le C\ep^{-(p-1)/\{1-\mu(p-1)\}} &&\mbox{if }\al>1. 
\end{align*}
See \cite[Theorem 2.1]{TW4} for the upper bounds of $T_\ep$ in the case $0\le \al<1$. 

It follows from these results above that in the special case $\mu=\al\ge 1$, the critical exponent $p_c$ is $p_c=1+1/\al$, which distinguishes between the existence and nonexistence of time-global solutions of \eqref{eq}, although the case $p=p_c$ is an open problem. Moreover, we see that if $1<p<1+1/\mu$ and $\mu=\al>1$, then the lifespan $T_\ep$ is estimated from above and below by the same power of $\ep$, thus 
\begin{equation}
T_\ep \sim \ep^{-(p-1)/\{1-\al(p-1)\}}. 
\label{Tep}
\end{equation}
We also observe that 
if $n=1$ and $\mu=\al< 1$,  then the critical exponent $p_c$ is determined exactly by $p_c=1+1/\al$ which is in the blow-up range, and  the lifespan has the sharp estimate which is the same as \eqref{Tep}. 

If the scale factor is $e^{Ht}$ which corresponds to the de Sitter spacetime, then there appears a pair of coefficients $(a(t),b(t))=(e^{-Ht},nH)$, where $H>0$ is  the Hubble constant. 
It follows from the theorem that there exist global solutions for small initial data and $p>1$. 
This result contrasts sharply with the case of the nonlinear terms $|u|^p$ or $|\nb_x u|^p$ as stated in Introduction. 


\section{Blow up}
\setcounter{equation}{0}

We consider the case where $b$ is a non-positive function in \eqref{eq} and $F(u_t)=|u_t|^p$, which gives rise to a generalization of the equation in the anti-de Sitter spacetime. 
To make the equation easier to see, we rewrite the problem as follows. 
\begin{equation}
\begin{cases}
\dsp
u_{tt}-a(t)^2\Delta u=b_-(t)u_t+|u_t|^p, &\qquad  t>0, \; x\in \R^n,   \\
  & \\
u(0,x)=\ep u_0(x), \; u_t(0,x)=\ep u_1(x), &\qquad x\in \R^n. 
\end{cases}
\label{eq-b}
\end{equation}
Another main result is the following theorem. 

\begin{theo}
Let $n\ge 1$ and $p>1$, and let 
$u_0\in C^2(\R^n)$ and $u_1\in C^1(\R^n)$,  $\mbox{\rm supp }u_0, \mbox{\rm supp }u_1\subset \{|x|\le R\}$ with $R>0$ and 
\[
\dsp \int_{\R^n} u_1(x)dx >0.
\] 
Assume that $a\in C^1([0,\infty))$ and $b_-\in C([0,\infty))$ satisfy $a(t), b_-(t)>0$ for $t\ge 0$, and 
\begin{equation}
\al\equiv\sup_{t\ge 0}\fa{na(t)}{A(t)+R}\le \inf_{t\ge 0}b_-(t)\equiv \bt,  
\label{albt} 
\end{equation}
where
\begin{equation}
A(t)=\int_0^t a(s)ds. 
\label{aA}
\end{equation}
Suppose that the problem \eqref{eq-b} has  a classical solution $u\in C^2([0,T)\times\R^n)$. 
Then $T<\infty$ and there exists a constant $\ep_0>0$ such that for $0<\ep\le \ep_0$, the lifespan $T_\ep$ of the solution $u$ to \eqref{eq-b} satisfies 
\begin{align*}
T_\ep &\le C\ep^{-(p-1)}\ln\fa 1\ep&& \mbox{if  }\bt=\al, \\
T_\ep &\le C\ln \fa 1\ep && \mbox{if  }\bt>\al,
\end{align*}
where $C>0$ is a constant independent of $\ep$.  

\label{th31}
\end{theo}

\begin{proof}

We first note from Corollary 4.2 in \cite{TW7} that $C^2$-solution $u$ of \eqref{eq-b} has the property of finite speed of propagation, and satisfies
\begin{equation}
\mbox{supp }u(t,\cdot)\subset \{|x|\le A(t)+R\}, 
\label{suppu0}
\end{equation}
where $A(t)$ is given by \eqref{aA}.

We prove the theorem in a similar manner to that in \cite{Na}. 
Set  
\begin{equation}
W(t)=\int u_t(t,x)dx \quad \mbox{and }\quad 
W_0=\int u_t(0,x)dx=\ep\int u_1(x)dx>0. 
\label{wdata}
\end{equation}
Integrating Eq \eqref{eq-b} over $\R^n$, we have 
\begin{equation}
W_t(t)=b_-W(t)+\int |u_t|^p dx 
\label{w1}
\end{equation}
by the divergence theorem. 
We see that $W_t\ge b_-W$, hence 
\begin{equation}
W(t)\ge W_0\exp\lp{\int_{0}^t b_-(s)ds}>0 \qquad \mbox{for }t\ge 0. 
\label{wexp}
\end{equation}
By H\"older's inequality and \eqref{suppu0}, 
\begin{equation}
\int|u_t|^pdx \ge (A(t)+R)^{-n(p-1)}W(t)^p. 
\label{Ho}
\end{equation}
From \eqref{w1}-\eqref{Ho}, 
\begin{align}
W_t&\ge b_-W+(A(t)+R)^{-n(p-1)}W^p 
\label{w_t}\\
 &\ge b_-W+(A(t)+R)^{-n(p-1)}W_0^{p-1}\exp\lp{(p-1)\int_{0}^t b_-(s)ds}W.  \nonumber
\end{align}

We observe from \eqref{aA} that 
\begin{equation}
\ln(A(t)+R)=\int_{0}^t \fa{a(s)}{A(s)+R}ds+\ln R. 
\label{ln}
\end{equation}
We then have
\begin{align}
W_t &\ge b_-W+R^{-n(p-1)}W_0^{p-1}\exp\lbt{(p-1)\int_{0}^t \lp{b_-(s)-\fa{na(s)}{A(s)+R}}ds}W.  \nonumber
\end{align}
Set $\dsp\gamma=R^{-n(p-1)}$. 
By assumption, we see that
\[
W_t(t)\ge (b_-(t)+\gamma W_0^{p-1})W \qquad \mbox{for }t\ge 0, 
\]
thus obtain 
\begin{equation} 
 W(t)\ge W_0\exp\lp{\int_{0}^t \lp{b_-(s)+\gamma W_0^{p-1}}ds}
 \qquad \mbox{for }t\ge 0. 
 \label{w}
\end{equation}
By \eqref{w_t} and \eqref{w}, taking $0<\delta <1$, we have 
\begin{align}
W_t &\ge (A(t)+R)^{-n(p-1)}W^p 
\label{WAtR}\\
 &\ge (A(t)+R)^{-n(p-1)}W_0^{(1-\dt)(p-1)} \nonumber\\
 & \qquad \cdot \exp\lb{(1-\dt)(p-1)\lp{\int_{0}^t \lp{b_-(s)+\gamma W_0^{p-1}}ds}}W^{1+\dt(p-1)}  \nonumber
\end{align} 
for $t\ge 0$ since $b_-(t)W(t)>0$. 
We determine $\delta$ below. 

Using \eqref{albt} and \eqref{ln}, we obtain
\begin{align}
W_t & \ge \gamma W_0^{(1-\delta)(p-1)} \nonumber \\
 & \quad \cdot \exp\lb{(p-1)\int_{0}^t \lp{(1-\dt)(b_-(s)+\gamma W_0^{p-1})-\fa{na(s)}{A(s)+R}}ds} \nonumber \\
 &\qquad \cdot W^{1+\dt(p-1)} \nonumber\\
& \ge \gamma W_0^{(1-\delta)(p-1)}
\exp\lb{(p-1)\int_{0}^t \lp{(1-\dt)(\bt+\gamma W_0^{p-1})-\al}ds}W^{1+\dt(p-1)} 
\label{Wt-0} 
\end{align}
for $t\ge 0$. 

By assumption $\bt\ge \al$ and $W_0>0$, we can choose $\dt>0$ so small that 
\[
(1-\dt)(\bt+\gamma W_0^{p-1})-\al>0. 
\]
In fact, how to choose $\dt$ depends on whether $\al=\bt$ or $\al<\bt$. 

\noindent
{\bf Case $\al=\bt$}. Let $\dt = c\gamma W_0^{p-1}/(\bt+\gamma W_0^{p-1})$ with some $0<c<1$. Then $(1-\dt)(\bt+\gamma W_0^{p-1})-\al=(1-c)\gamma W_0^{p-1}$. Hence from \eqref{Wt-0}, we have 
\[
W_t \ge \gamma W_0^{(1-\delta)(p-1)}
\exp\lb{(p-1)(1-c)\gamma W_0^{p-1}t}W^{1+\dt(p-1)}
\]
for $t\ge 0$, which shows that $W(t)$ blows up in a finite time for $p>1$.
It follows from simple calculation that the existence time $T$ satisfies $T\lesssim \ep^{-(p-1)}\ln \ep^{-1}$.

\noindent
{\bf Case $\al<\bt$}. Let $\dt = (\bt-\al)/(2\bt)$. Then $(1-\dt)(\bt+\gamma W_0^{p-1})-\al>(\bt-\al)/2$. Hence from \eqref{Wt-0}, we have 
\[
W_t \ge \gamma W_0^{(1-\delta)(p-1)}
\exp\lb{(p-1)(\bt-\al)t/2}W^{1+\dt(p-1)}
\]
for $t\ge 0$, which shows that $W(t)$ blows up in a finite time for $p>1$.
It follows from simple calculation that $T\lesssim \ln \ep^{-1}$.
This completes the proof of Theorem \ref{th31}. 
\end{proof}

As a corollary, we obtain the following result from \eqref{WAtR}. 
\begin{coro}
Let $n,p, u_0$ and $u_1$ be as in Theorem \ref{th31}. 
Assume that $a\in C^1([0,\infty))$ and $b_-\in C([0,\infty))$ satisfy $a(t), b_-(t)>0$ for $t\ge 0$, and 
$A(t)^{-n(p-1)}\not\in L^1([0,\infty))$, 
where $A(t)$ is given by \eqref{aA}. 
Suppose that the problem \eqref{eq-b} has  a classical solution $u\in C^2([0,T)\times\R^n)$. 
Then $T<\infty$ and there exists a constant $\ep_0>0$ such that for $0<\ep\le \ep_0$, 
$T$ satisfies 
\[
\int_0^T (A(t)+R)^{-n(p-1)}dt \le C\ep^{-(p-1)}. 
\]
In particular, if $a(t)=(t+1)^{\al}$ with $\al\in\R$, then the lifespan $T_\ep$ of the solution $u$ to \eqref{eq-b} satisfies 
\begin{align*}
T_\ep &\le C\ep^{-(p-1)/\{1-n(p-1)(\al+1)\}}&& \mbox{if  }n(p-1)(\al+1)<1, \\
T_\ep &\le \exp \lp{C\ep^{-(p-1)}} && \mbox{if  }n(p-1)(\al+1)=1,
\end{align*}
where $C>0$ is a constant independent of $\ep$.  

\label{cor32}
\end{coro}

In the end of this subsection, we discuss specific coefficients $a(t)$ and $b_-(t)$ which are applicable to Theorem \ref{th31} by considering the equation in the anti-de Sitter spacetime.

If the scale factor is $e^{-Ht}$ which corresponds to the anti-de Sitter spacetime, then there appears a pair of coefficients $(a(t),b_-(t))=(e^{Ht},nH)$, where $H>0$ is  the Hubble constant. 
Let $HR\ge 1$ with $R$ given in Theorem \ref{th31}. We then see that $\al=\bt=nH$ in Theorem \ref{th31}, so that 
the lifespan satisfies $T_\ep\le C\ep^{-(p-1)}\ln 1/\ep$.

Finally, if the scale factor of the contracting FLRW metric is $(t+1)^{-\al}$, then 
$\Box_g u=-F(u_t)$ is of the form \eqref{eq-b} with $(a(t),b_-(t))=((t+1)^\al,\mu/(t+1))$, where $\al>0$ and $\mu\ge 0$ is a constant. 
This case reduces to the generalized Tricomi equations.
We refer to \cite{LP,LS} for the Tricomi equations.


\end{document}